\documentclass[12pt]{amsart}
\usepackage{amsmath,amsthm,amsfonts,amssymb,mathrsfs}
\date{\today}

\usepackage[cp1251]{inputenc}         

\usepackage[ukrainian]{babel} 
\usepackage{amsmath,amsthm,amsfonts,amssymb,mathrsfs}
\usepackage{eucal}

\usepackage{color}

\usepackage{amssymb,cite}
\usepackage[colorlinks,plainpages,citecolor=magenta, linkcolor=blue, backref]{hyperref}

\usepackage{hyperref}

\newtheorem{theorem}{Теорема}

\newtheorem{proposition}{Твердження}
\newtheorem{corollary}{Наслiдок}
\newtheorem{lemma}{Лема}
\theoremstyle{definition}
\newtheorem{example}{Приклад}
\newtheorem{remark}{Зауваження}

\newtheorem{definition}{Означення}

\begin{document}

\title[Про напiвгрупу, породжену розширеною бiциклiчною напiвгрупою ]{Про напiвгрупу, породжену розширеною бiциклiчною напiвгрупою та $\omega$-замкненою сiм'Єю}

\author[Олег~Гутік, Інна Позднякова]{Олег~Гутік, Інна Позднякова}
\address{Львівський національний університет ім. Івана Франка, Університецька 1, Львів, 79000, Україна}
\email{oleg.gutik@lnu.edu.ua,
ovgutik@yahoo.com, pozdnyakova.inna@gmail.com}



\begin{abstract}
Введено поняття алгебраїчне розширення $\boldsymbol{B}_{\mathbb{Z}}^{\mathscr{F}}$ розширеної біциклічної напівгрупи для довільної $\omega$-замк\-не\-ної сім'ї $\mathscr{F}$ підмножин в $\omega$. Доведено, що $\boldsymbol{B}_{\mathbb{Z}}^{\mathscr{F}}$ є комбінаторною інверсною напівгрупою. Описано відношення Ґріна, природний частковий порядок на напівгрупі $\boldsymbol{B}_{\mathbb{Z}}^{\mathscr{F}}$ та її множину ідемпотентів. Знайдено критерії  простоти, $0$-простоти, біпростоти та $0$-біпростоти напівгрупи $\boldsymbol{B}_{\mathbb{Z}}^{\mathscr{F}}$, а також коли напівгрупа $\boldsymbol{B}_{\mathbb{Z}}^{\mathscr{F}}$  ізоморфна розширеній біциклічній напівгрупі або зліченній напівгрупі матричних одиниць. Доведено, що у випадку, коли сім'я $\mathscr{F}$ складається з усіх одноточокових підмножини в $\omega$ та порожньої множини, то напівгрупа $\boldsymbol{B}_{\mathbb{Z}}^{\mathscr{F}}$  ізоморфна $\lambda$-розширенню Брандта напівґратки $(\omega,\min)$.

\bigskip
\noindent
\emph{Oleg Gutik, Inna Pozdnyakova, \textbf{On the semigroup generating by extended bicyclic semigroup and a $\omega$-closed family}.}

\smallskip
\noindent
The algebraic extension $\boldsymbol{B}_{\mathbb{Z}}^{\mathscr{F}}$ of the extended bicyclic semigroup for an arbitrary $\omega$-closed family $\mathscr{F}$ subsets of $\omega$ is introduced. It is proven that $\boldsymbol{B}_{\mathbb{Z}}^{\mathscr{F}}$ is a combinatorial inverse semigroup. Green's relations, the natural partial order on the semigroup $\boldsymbol{B}_{\mathbb{Z}}^{\mathscr{F}}$ and its set of idempotents are described. The criteria of simplicity, $0$-simplicity, bisimplicity, $0$-bisimplicity of the semigroup $\boldsymbol{B}_{\mathbb{Z}}^{\mathscr{F}}$ and the criterion for $\boldsymbol{B}_{\mathbb{Z}}^{\mathscr{F}}$ to be isomorphic to the extended bicyclic semigroup or the countable semigroup of matrix units are derived. It is proved that in the case when the family $\mathscr{F}$ consists of all singletons of $\omega$ and the empty set, the semigroup $\boldsymbol{B}_{\mathbb{Z}}^{\mathscr{F}}$ is isomorphic to the Brandt $\lambda$-extension of the semilattice $(\omega,\min)$.
\end{abstract}

\maketitle


\section{Вступ}\label{section-1}

У цій роботі будемо користуватися термінологією з \cite{Clifford-Preston-1961, Clifford-Preston-1967, Lawson-1998, Petrich-1984}.
Множину цілих чисел  позначатимемо через $\mathbb{Z}$, а множину невід'ємних цілих чисел  позначатимемо через $\omega$. Для довільного $k\in\mathbb{Z}$ позначимо $[k)=\{i\in\mathbb{Z}\colon i\geqslant k\}$.

Підмножину $A$ в $\omega$ називають \emph{індуктивною}, якщо з того, що $i\in A$ випливає, що $i+1\in A$. Очевидно, що $\varnothing$ --- індуктивна множина в $\omega$, і непорожня підмножина $A\subseteq\omega$ є індук\-тивною тоді і лише тоді, коли $A=[k)$ для деякого $k\in\omega$.

Якщо $S$~--- напівгрупа, то її підмножина ідемпотентів позначається через $E(S)$.  На\-пів\-гру\-пу $S$ називають \emph{інверсною}, якщо для довільного її елемента $x$ існує єдиний елемент $x^{-1}\in S$ такий, що $xx^{-1}x=x$ та $x^{-1}xx^{-1}=x^{-1}$ \cite{Petrich-1984, Wagner-1952}. В інверсній напівгрупі $S$ вище означений елемент $x^{-1}$ називається \emph{інверсним до} $x$. \emph{В'язка}~--- це напівгрупа ідемпотентів, а \emph{напівґратка}~--- це комутативна в'язка.

Якщо $S$ --- напівгрупа, то ми позначатимемо відношення Ґріна на $S$ через $\mathscr{R}$, $\mathscr{L}$, $\mathscr{D}$, $\mathscr{H}$ і $\mathscr{J}$ (див. означення в \cite[\S2.1]{Clifford-Preston-1961} або \cite{Green-1951}). Напівгрупа $S$ називається \emph{простою}, якщо $S$ не містить власних двобічних ідеалів, тобто $S$ складається з одного $\mathscr{J}$-класу, і \emph{біпростою}, якщо $S$ складається з одного $\mathscr{D}$-класу.

Відношення еквівалентності $\mathfrak{K}$ на напівгрупі $S$ називається \emph{конгруенцією}, якщо для елементів $a$ та $b$ напівгрупи $S$ з того, що виконується умова $(a,b)\in\mathfrak{K}$ випливає, що $(ca,cb), (ad,bd) \in\mathfrak{K}$, для довільних $c,d\in S$. Відношення $(a,b)\in\mathfrak{K}$ ми також будемо записувати $a\mathfrak{K}b$, і в цьому випадку будемо говорити, що \emph{елементи $a$ i $b$ є $\mathfrak{K}$-еквівалентними}.

Якщо $S$~--- напівгрупа, то на $E(S)$ визначено частковий порядок:
$
e\preccurlyeq f
$   тоді і лише тоді, коли
$ef=fe=e$.
Так означений частковий порядок на $E(S)$ називається \emph{при\-род\-ним}.

Означимо відношення $\preccurlyeq$ на інверсній напівгрупі $S$ так:
$
    s\preccurlyeq t
$
тоді і лише тоді, коли $s=te$.
для деякого ідемпотента $e\in S$. Так означений частковий порядок назива\-єть\-ся \emph{при\-род\-ним част\-ковим порядком} на інверсній напівгрупі $S$~\cite{Wagner-1952}. Очевидно, що звуження природного часткового порядку $\preccurlyeq$ на інверсній напівгрупі $S$ на її в'язку $E(S)$ є при\-род\-ним частковим порядком на $E(S)$.

Нехай $\lambda$ --- довільний ненульовий кардинал. Відображення $\alpha$ за підмножини $D\subseteq\lambda$ в кардинал $\lambda$ називається \emph{частковим перетворенням} кардинала $\lambda$. У цьому випадку множина $D$ називається \emph{областю визначення} часткового перетворення $\alpha$ і позначається $\operatorname{dom}\alpha$. Образ елемента $x\in\operatorname{dom}\alpha$ стосовно $\alpha$ ми будемо  позначати через $x\alpha$. Множина $\{ x\in \lambda\colon y\alpha=x \mbox{ для деякого } y\in\lambda\}$ називається \emph{областю значень} часткового перетворення $\alpha$ і позначається $\operatorname{ran}\alpha$. Для зручності позначимо через $\varnothing$ порожнє перетворення, тобто таке часткове перетворення з $\operatorname{dom}\varnothing=\operatorname{ran}\varnothing=\varnothing$.

Нехай $\mathscr{I}_\lambda$ --- множина всіх взаємно однозначних часткових перетворень кардинала $\lambda$ разом з такою напівгруповою операцією:
\begin{equation*}
    x(\alpha\beta)=(x\alpha)\beta \quad \mbox{якщо} \quad
    x\in\operatorname{dom}(\alpha\beta)=\{
    y\in\operatorname{dom}\alpha\colon
    y\alpha\in\operatorname{dom}\beta\}, \qquad \mbox{для} \quad
    \alpha,\beta\in\mathscr{I}_\lambda.
\end{equation*}
Напівгрупа $\mathscr{I}_\lambda$ називається \emph{симетричною інверсною напівгрупою} (\emph{симетричним інверсним моноїдом}) над кардиналом $\lambda$~(див. \cite{Clifford-Preston-1961}). Симетрична інверсна напівгрупа введена В. В.~Вагнером~\cite{Wagner-1952} і вона відіграє важливу роль у теорії напівгруп.

Нагадаємо (див.  \cite[\S1.12]{Clifford-Preston-1961}, що \emph{біциклічною напівгрупою} (або \emph{біциклічним моноїдом}) ${\mathscr{C}}(p,q)$ називається напівгрупа з одиницею, породжена двоелементною мно\-жи\-ною $\{p,q\}$ і визначена одним  співвідношенням $pq=1$. Біциклічна на\-пів\-група відіграє важливу роль у теорії
на\-півгруп. Так, зокрема, класична теорема О.~Ан\-дерсена \cite{Andersen-1952}  стверджує, що {($0$-)}прос\-та напівгрупа з (ненульовим) ідем\-по\-тен\-том є цілком {($0$-)}прос\-тою тоді і лише тоді, коли вона не містить ізоморфну копію бі\-циклічного моноїда. Різні розширення біциклічного моноїда вводилися раніше різ\-ни\-ми авторами \cite{Fortunatov-1976, Fotedar-1974, Fotedar-1978, Warne-1967}. Такими є, зокрема, конструкції Брука та Брука--Рейлі занурення напівгруп у прості та описання інверсних біпростих і $0$-біпростих $\omega$-напівгруп \cite{Bruck-1958, Reilly-1966, Warne-1966, Gutik-2018}.

\begin{remark}\label{remark-1.1}
Легко бачити, що біциклічний моноїд ${\mathscr{C}}(p,q)$ ізоморфний напівгрупі, заданій на множині $\boldsymbol{B}_{\omega}=\omega\times\omega$ з напівгруповою операцією
\begin{equation}\label{eq-1.1}
  (i_1,j_1)\cdot(i_2,j_2)=(i_1+i_2-\min\{j_1,i_2\},j_1+j_2-\min\{j_1,i_2\})=
\left\{
  \begin{array}{ll}
    (i_1-j_1+i_2,j_2), & \hbox{якщо~} j_1<i_2;\\
    (i_1,j_2), & \hbox{якщо~} j_1=i_2;\\
    (i_1,j_1-i_2+j_2), & \hbox{якщо~} j_1>i_2.
  \end{array}
\right.
\end{equation}
\end{remark}

Множина $\boldsymbol{B}_{\mathbb{Z}}=\mathbb{Z}\times\mathbb{Z}$ з напівгруповою операцією \eqref{eq-1.1} називається \emph{розширеною біциклічною напівгрупою} \cite{Warne-1968}. Очевидно, шо $\boldsymbol{B}_{\omega}$ --- піднапівгрупа напівгрупи $\boldsymbol{B}_{\mathbb{Z}}$.

\begin{remark}\label{remark-1.2}
Нехай $\alpha$ i $\beta$~--- часткові перетворення множини $\omega$, які визначаються так:
\begin{align*}
  &\operatorname{dom}\alpha=\omega,  &&\operatorname{ran}\alpha=\omega\setminus\{0\}, &&(n)\alpha=n+1, \\
  &\operatorname{dom}\beta=\omega\setminus\{0\}, &&\operatorname{ran}\beta=\omega, &&(n+1)\beta=n.
\end{align*}

Легко бачити, що біциклічний моноїд ${\mathscr{C}}(p,q)$ ізоморфний напівгрупі $\mathscr{C}_\omega$, яка породжена елементами $\alpha$ i $\beta$ (див. \cite[вправа~IV.1.11$(ii)$]{Petrich-1984}).
\end{remark}

\begin{definition}\label{definition-1.3}
Для довільних $i,j\in\mathbb{Z}$ означимо часткове перетворення $\alpha^i_j\colon\mathbb{Z}\rightharpoonup\mathbb{Z}$ так:
\begin{equation*}
  \operatorname{dom}\alpha^i_j=[i), \qquad \operatorname{ran}\alpha^i_j=[j) \qquad \hbox{i} \qquad (i+n)\alpha^i_j=j+n, \quad \hbox{для всіх} \quad i+n\in\operatorname{dom}\alpha^i_j.
\end{equation*}
Тоді множина $\boldsymbol{B}(\mathbb{Z})=\left\{\alpha^i_j\colon i,j\in\mathbb{Z}\right\}$ стосовно операції композиції часткових перетворень множини цілих чисел $\mathbb{Z}$ є інверсною напівгрупою (див. \cite{Fortunatov-1976, Fotedar-1974, Fotedar-1978, Gutik-Pagon-Pavlyk-2011}), а отже є інверсною піднапівгрупою симетричного інверсного моноїда $\mathscr{I}_\omega$.
\end{definition}

Напівгрупова операція на $\boldsymbol{B}(\mathbb{Z})$ визначається за формулою
\begin{equation*}
  \alpha^{i_1}_{j_1}\circ\alpha^{i_2}_{j_2}=\alpha^i_j, \quad \hbox{де} \quad i=i_1+i_2-\min\{j_1,i_2\} \quad \hbox{i} \quad j=j_1+j_2-\min\{j_1,i_2\}.
\end{equation*}
(див. \cite[формула (1)]{Gutik-Pagon-Pavlyk-2011}). Отже, виконується

\begin{proposition}\label{proposition-1.4}
Відображення $\mathfrak{h}\colon \boldsymbol{B}(\mathbb{Z})\to \boldsymbol{B}_{\mathbb{Z}}$, $\alpha^i_j\mapsto (i,j)$ є ізоморфізмом.
\end{proposition}

\section{Конструкція напівгрупи $\boldsymbol{B}_\mathbb{Z}^\mathscr{F}$}\label{section-2}

Нехай $\mathscr{P}(\omega)$~--- сім'я усіх підмножин ординала $\omega$. Для довільних $F\in\mathscr{P}(\omega)$ i $n,m\in\mathbb{Z}$ покладемо $n-m+F=\{n-m+k\colon k\in F\}$, якщо $F\neq\varnothing$, i  $n-m+F=\varnothing$ при $F=\varnothing$.
Будемо говорити, що підсім'я  $\mathscr{F}\subseteq\mathscr{P}(\omega)$ є \emph{${\omega}$-замкненою}, якщо $F_1\cap(-n+F_2)\in\mathscr{F}$ для довільних $n\in\omega$ i $F_1,F_2\in\mathscr{F}$.

Нехай $\alpha^{i_1}_{j_1},\alpha^{i_2}_{j_2}$~--- довільні елементи напівгрупи $\boldsymbol{B}(\mathbb{Z})$ та $\overline{F}_1$ i $\overline{F}_2$~--- довільні підмножини в $\operatorname{dom}\alpha^{i_1}_{j_1}$ i $\operatorname{dom}\alpha^{i_2}_{j_2}$, відповідно. Позначимо через $\alpha^{i_1}_{j_1}|_{\overline{F}_1}$ i $\alpha^{i_2}_{j_2}|_{\overline{F}_2}$ звуження часткових перетворень $\alpha^{i_1}_{j_1}$ i $\alpha^{i_2}_{j_2}$ на множини $\overline{F}_1$ i $\overline{F}_2$, відповідно. Тоді з означення композиції часткових перетворень випливає, що
\begin{align*}
   \operatorname{dom}\left(\alpha^{i_1}_{j_1}|_{\overline{F}_1}\circ\alpha^{i_2}_{j_2}|_{\overline{F}_2}\right)
   &=\operatorname{dom}\left(\alpha^{i_1}_{j_1}|_{\overline{F}_1}\right)\cap \left(\overline{F}_2\right)\left(\alpha^{i_1}_{j_1}|_{\overline{F}_1}\right)^{-1}=
   \overline{F}_1\cap\left(\overline{F}_2\right)\left(\alpha^{i_1}_{j_1}|_{\overline{F}_1}\right)^{-1}=\\
   &=\overline{F}_1\cap \left(\overline{F}_2\right)\left(\alpha^{j_1}_{i_1}|_{(\overline{F}_1)\alpha^{i_1}_{j_1}}\right)=
   \overline{F}_1\cap \left(\left(\overline{F}_2\right)\alpha^{j_1}_{i_1}\cap\overline{F}_1\right)=
   \overline{F}_1\cap \left(\overline{F}_2\right)\alpha^{j_1}_{i_1}=\\
   &=\overline{F}_1\cap \left(\overline{F}_2\right)\left(\alpha^{i_1}_{j_1}\right)^{-1}.
\end{align*}

Позаяк $\overline{F}_1\subseteq \operatorname{dom}\alpha^{i_1}_{j_1}$ i $\overline{F}_2\subseteq \operatorname{dom}\alpha^{i_2}_{j_2}$, то існують множини $F_1,F_2\in \mathscr{P}(\omega)$ такі, що $\overline{F}_1=i_1+F_1$ i $\overline{F}_2=i_2+F_2$. Тоді
\begin{align*}
   \operatorname{dom}\left(\alpha^{i_1}_{j_1}|_{\overline{F}_1}\circ\alpha^{i_2}_{j_2}|_{\overline{F}_2}\right)
   &=\overline{F}_1\cap \left(\overline{F}_2\right)\left(\alpha^{i_1}_{j_1}\right)^{-1}=
   \left(i_1+F_1\right)\cap \left(i_2+F_2\right)\left(\alpha^{i_1}_{j_1}\right)^{-1}=\\
   &=\left\{
  \begin{array}{ll}
    i_1-j_1+i_2+((j_1-i_2+F_1)\cap F_2), & \hbox{якщо~~} j_1<i_2;\\
    i_1+(F_1\cap F_2),                   & \hbox{якщо~~} j_1=i_2;\\
    i_1+(F_1\cap (i_2-j_1+F_2)),         & \hbox{якщо~~} j_1>i_2.
  \end{array}
\right.
\end{align*}

Нехай $\boldsymbol{B}_{\mathbb{Z}}$~--- розширена біциклічна напівгрупа та $\mathscr{F}$ --- ${\omega}$-замкнена підсім'я в  $\mathscr{P}(\omega)$. На множині $\boldsymbol{B}_{\mathbb{Z}}\times\mathscr{F}$ означимо бінарну операцію ``$\cdot$'' за формулою
\begin{equation}\label{eq-2.1}
  (i_1,j_1,F_1)\cdot(i_2,j_2,F_2)=
  \left\{
    \begin{array}{ll}
      (i_1-j_1+i_2,j_2,(j_1-i_2+F_1)\cap F_2), & \hbox{якщо~~} j_1<i_2;\\
      (i_1,j_2,F_1\cap F_2)                    & \hbox{якщо~~} j_1=i_2;\\
      (i_1,j_1-i_2+j_2,F_1\cap (i_2-j_1+F_2)), & \hbox{якщо~~} j_1>i_2.
    \end{array}
  \right.
\end{equation}

\begin{remark}
Якщо $\boldsymbol{B}_{\omega}$~--- біциклічний моноїд і $\mathscr{F}$ --- ${\omega}$-замкнена підсім'я в  $\mathscr{P}(\omega)$, то на множині $\boldsymbol{B}_{\omega}\times\mathscr{F}$ бінарна операція \eqref{eq-2.1} є асоціативною (див. твердження 1 \cite{Gutik-Mykhalenych-2020}).
\end{remark}

Оскільки довільні три елементи $(i_1,j_1)$, $(i_2,j_2)$, $(i_3,j_3)$ розширеної біциклічної напівгрупи $\boldsymbol{B}(\mathbb{Z})$ та їх добуток містяться в її піднапівгрупі  $\boldsymbol{B}^k(\mathbb{Z})=\left\{(i,j)\in\boldsymbol{B}(\mathbb{Z})\colon i,j\geqslant k\right\}$, де $k=\min\{i_1,j_1,i_2,j_2, i_3,j_3\}$, яка за твердженням 2.11$(viii)$~\cite{Fihel-Gutik-2011} ізоморфна біциклічному моноїду, то з твердження 1 в \cite{Gutik-Mykhalenych-2020} випливає, що бінарна операція \eqref{eq-2.1} на $\boldsymbol{B}_{\mathbb{Z}}\times\mathscr{F}$ є асоціативною, а отже виконується

\begin{proposition}\label{proposition-2.1}
Якщо сім'я  $\mathscr{F}\subseteq\mathscr{P}(\omega)$ є ${\omega}$-замкненою, то $(\boldsymbol{B}_{\mathbb{Z}}\times\mathscr{F},\cdot)$ є напівгрупою.
\end{proposition}

Припустимо, що ${\omega}$-замкнена сім'я $\mathscr{F}\subseteq\mathscr{P}(\omega)$ містить порожню множину $\varnothing$, то з означення напівгрупової операції в $(\boldsymbol{B}_{\mathbb{Z}}\times\mathscr{F},\cdot)$ випливає, що множина
$ 
  \boldsymbol{I}=\{(i,j,\varnothing)\colon i,j\in\mathbb{Z}\}
$ 
є ідеалом напівгрупи $(\boldsymbol{B}_{\mathbb{Z}}\times\mathscr{F},\cdot)$.

\begin{definition}\label{definition-2.2}
Для довільної ${\omega}$-замкненої сім'ї $\mathscr{F}\subseteq\mathscr{P}(\omega)$ означимо
\begin{equation*}
  \boldsymbol{B}_{\mathbb{Z}}^{\mathscr{F}}=
\left\{
  \begin{array}{ll}
    (\boldsymbol{B}_{\mathbb{Z}}\times\mathscr{F},\cdot)/\boldsymbol{I}, & \hbox{якщо~} \varnothing\in\mathscr{F};\\
    (\boldsymbol{B}_{\mathbb{Z}}\times\mathscr{F},\cdot),                & \hbox{якщо~} \varnothing\notin\mathscr{F}.
  \end{array}
\right.
\end{equation*}
\end{definition}

\begin{remark}
Очевидно, що для довільної ${\omega}$-замкненої сім'ї $\mathscr{F}\subseteq\mathscr{P}(\omega)$ напівгрупа  $\boldsymbol{B}_{\omega}^{\mathscr{F}}$, введена в \cite{Gutik-Mykhalenych-2020}, є інверсною піднапівгрупою в $\boldsymbol{B}_{\mathbb{Z}}^{\mathscr{F}}$.
\end{remark}

У наступному параграфі за аналогією з \cite{Gutik-Mykhalenych-2020}, вивчаються властивості алгебраїчного розширення  $\boldsymbol{B}_{\mathbb{Z}}^{\mathscr{F}}$ розширеної біциклічної напівгрупи для довільної $\omega$-замкненої сім'ї $\mathscr{F}$ підмножин в $\omega$.
\section{Алгебраїчні властивості напівгрупи $\boldsymbol{B}_{\mathbb{Z}}^{\mathscr{F}}$}\label{section-3}

Надалі будемо вважати, що $\mathscr{F}$ --- ${\omega}$-замкнена підсім'я в  $\mathscr{P}(\omega)$.

Доведення тверджень наступної леми проводяться звичайною перевіркою, аналогічно відповідним твердженням в \S~3 з \cite{Gutik-Mykhalenych-2020}.

\begin{lemma}\label{lemma-3.1}
Нехай $\mathscr{F}$ --- $\omega$-замкнена сім'я в $\mathscr{P}(\omega)$. Тоді:
\begin{enumerate}
  \item\label{lemma-3.1(1)} Напівгрупа $\boldsymbol{B}_{\mathbb{Z}}^{\mathscr{F}}$ містить нуль  $\boldsymbol{0}$ тоді і лише тоді, коли $\varnothing\in\mathscr{F}$, причому нуль $\boldsymbol{0}$ є образом ідеала $\boldsymbol{I}$ при природному гомоморфізмі, породженому конгруенцією Ріса
\begin{equation*}
\boldsymbol{\mathfrak{C}}_{\boldsymbol{I}}=\left\{(x,x)\colon x\in\boldsymbol{B}_{\mathbb{Z}}\times\mathscr{F} \right\}\cup(\boldsymbol{I}\times\boldsymbol{I})
\end{equation*}
на напівгрупі $(\boldsymbol{B}_{\mathbb{Z}}\times\mathscr{F},\cdot)$.
  \item\label{lemma-3.1(2)} Ненульовий елемент $(i,j,F)$ напівгрупи $\boldsymbol{B}_{\mathbb{Z}}^{\mathscr{F}}$ є ідемпотентом тоді і лише тоді, коли $i=j$.
  \item\label{lemma-3.1(3)} Ідемпотенти в $\boldsymbol{B}_{\mathbb{Z}}^{\mathscr{F}}$ комутують.
  \item\label{lemma-3.1(4)} Елементи $(i,j,F)$ і $(j,i,F)$ є інверсними в  $\boldsymbol{B}_{\mathbb{Z}}^{\mathscr{F}}$.
\end{enumerate}
\end{lemma}

З теореми 1.17 \cite{Clifford-Preston-1961} та тверджень \eqref{lemma-3.1(3)} і \eqref{lemma-3.1(4)} леми \ref{lemma-3.1} випливає

\begin{theorem}\label{theorem-3.2}
Якщо $\mathscr{F}$ --- ${\omega}$-замкнена підсім'я в  $\mathscr{P}(\omega)$, то $\boldsymbol{B}_{\mathbb{Z}}^{\mathscr{F}}$ --- інверсна напівгрупа.
\end{theorem}

Наступне твердження описує природний частковий порядок на напівгрупі $\boldsymbol{B}_{\mathbb{Z}}^{\mathscr{F}}$.

\begin{proposition}\label{proposition-3.3}
Нехай $(i_1,j_1,F_1)$ i $(i_2,j_2,F_2)$~--- ненульові елементи напівгрупи $\boldsymbol{B}_{\mathbb{Z}}^{\mathscr{F}}$. Тоді  \linebreak $(i_1,j_1,F_1)\preccurlyeq(i_2,j_2,F_2)$ тоді і лише тоді, коли $F_1\subseteq -k+F_2$ й $i_1-i_2=j_1-j_2=k$ для деякого $k\in\omega$.
\end{proposition}

\begin{proof}
$(\Longleftarrow)$ Якщо $F_1\subseteq -k+F_2$ та $i_1-i_2=j_1-j_2=k$ для деякого $k\in\omega$, то
\begin{align*}
  (i_1,j_1,F_1)&\cdot(i_1,j_1,F_1)^{-1}\cdot(i_2,j_2,F_2)=(i_1,j_1,F_1)\cdot(j_1,i_1,F_1)\cdot(i_2,j_2,F_2)=
   (i_1,i_1,F_1)\cdot(i_2,j_2,F_2)=\\
   &=(i_1,i_1-i_2+j_2,F_1\cap (i_2-i_1+F_2))
   =(i_1,j_1-j_2+j_2,F_1\cap (-k+F_2))=\\
   &   =(i_1,j_1,F_1\cap (-k+F_2))=(i_1,j_1,F_1),
\end{align*}
і за лемою~1.4.6 \cite{Lawson-1998} отримуємо, що $(i_1,j_1,F_1)\preccurlyeq(i_2,j_2,F_2)$ в $\boldsymbol{B}_{\mathbb{Z}}^{\mathscr{F}}$.

\smallskip
$(\Longrightarrow)$ Припустимо, що $(i_1,j_1,F_1)\preccurlyeq(i_2,j_2,F_2)$ в $\boldsymbol{B}_{\mathbb{Z}}^{\mathscr{F}}$. Згідно з лемою~1.4.6 \cite{Lawson-1998},
\begin{align*}
  (i_1,j_1,F_1)&=(i_1,j_1,F_1)\cdot(i_1,j_1,F_1)^{-1}\cdot(i_2,j_2,F_2)
   =(i_1,j_1,F_1)\cdot(j_1,i_1,F_1)\cdot(i_2,j_2,F_2)= \\
   &=(i_1,i_1,F_1)\cdot(i_2,j_2,F_2)=\\
   &=\left\{
    \begin{array}{ll}
      (i_1-i_1+i_2,j_2,(i_1-i_2+F_1)\cap F_2), & \hbox{якщо~} i_1<i_2;\\
      (i_1,j_2,F_1\cap F_2),                   & \hbox{якщо~} i_1=i_2;\\
      (i_1,i_1-i_2+j_2,F_1\cap (i_2-i_1+F_2)), & \hbox{якщо~} i_1>i_2
    \end{array}
  \right.  =\\
   &=\left\{
    \begin{array}{lll}
      (i_2,j_2,(i_1-i_2+F_1)\cap F_2),        & \hbox{якщо~} i_1<i_2;        &(1_1)\\
      (i_1,j_2,F_1\cap F_2),                  & \hbox{якщо~} i_1=i_2;        &(2_1)\\
      (i_1,i_1-i_2+j_2,F_1\cap(i_2-i_1+F_2)), & \hbox{якщо~} i_1>i_2         &(3_1)
    \end{array}
  \right.
\end{align*}
 
\begin{align*}
  (i_1,j_1,F_1)&=(i_2,j_2,F_2)\cdot(i_1,j_1,F_1)^{-1}\cdot(i_1,j_1,F_1)
   =(i_2,j_2,F_2)\cdot(j_1,i_1,F_1)\cdot(i_1,j_1,F_1)= \\
   &=(i_2,j_2,F_2)\cdot(j_1,j_1,F_1)=\\
   &=\left\{
    \begin{array}{ll}
      (i_2,j_2-j_1+j_1,F_2\cap (j_1-j_2+F_1)), & \hbox{якщо~} j_1<j_2;\\
      (i_2,j_1,F_2\cap F_1),                   & \hbox{якщо~} j_1=j_2;\\
      (i_2-j_2+j_1,j_1,(j_2-j_1+F_2)\cap F_1), & \hbox{якщо~} j_1>j_2
    \end{array}
  \right.  =\\
   &=\left\{
    \begin{array}{lll}
      (i_2,j_2,F_2\cap (j_1-j_2+F_1)),         & \hbox{якщо~} j_1<j_2;        &(1_2)\\
      (i_2,j_1,F_2\cap F_1),                   & \hbox{якщо~} j_1=j_2;        &(2_2)\\
      (i_2-j_2+j_1,j_1,(j_2-j_1+F_2)\cap F_1), & \hbox{якщо~} j_1>j_2.        &(3_2)
    \end{array}
  \right.
\end{align*}

У випадках $(1_1)$ i $(1_2)$ та $(2_1)$ i $(2_2)$ отримуємо, що $i_1=i_2$, $j_1=j_2$, а отже $i_1-i_2=j_1-j_2=0$ i $F_1\subseteq -0+F_2$. У випадках $(3_1)$ i $(3_2)$ отримуємо, що $i_1=i_2-j_2+j_1$, і врахувавши, що $i_1>i_2$, $j_1>j_2$ і $(j_2-j_1+F_2)\cap F_1=F_1$, та прийнявши $k=i_1-i_2=j_1-j_2$, отримуємо, що $F_1\subseteq -k+F_2$ для деякого $k\in\omega$.
\end{proof}


\begin{corollary}\label{corollary-3.4}
$(i,i,F_1)\preccurlyeq(j,j,F_2)$ в $E(\boldsymbol{B}_{\mathbb{Z}}^{\mathscr{F}})$ тоді і тільки тоді, коли $i\geqslant j$ i $F_1\subseteq j-i+F_2$.
\end{corollary}

Наступна теорема описує відношення Ґріна на напівгрупі $\boldsymbol{B}_{\omega}^{\mathscr{F}}$.

\begin{theorem}\label{theorem-3.5}
Нехай $\mathscr{F}$ --- ${\omega}$-замкнена підсім'я в  $\mathscr{P}(\omega)$ i $(i_1,j_1,F_1),(i_2,j_2,F_2)\in \boldsymbol{B}_{\mathbb{Z}}^{\mathscr{F}}$. Тоді:
\begin{itemize}
  \item[$\boldsymbol{(i)}$] $(i_1,j_1,F_1)\mathscr{R}(i_2,j_2,F_2)$ тоді і лише тоді, коли $i_1=i_2$ i $F_1=F_2$;
  \item[$\boldsymbol{(ii)}$] $(i_1,j_1,F_1)\mathscr{L}(i_2,j_2,F_2)$ тоді і лише тоді, коли $j_1=j_2$ i $F_1=F_2$;
  \item[$\boldsymbol{(iii)}$] $(i_1,j_1,F_1)\mathscr{H}(i_2,j_2,F_2)$ тоді і лише тоді, коли $i_1=i_2$, $j_1=j_2$ i $F_1=F_2$, а отже всі $\mathscr{H}$-класи напівгрупи $\boldsymbol{B}_{\omega}^{\mathscr{F}}$ є одноелементними;
  \item[$\boldsymbol{(iv)}$] $(i_1,j_1,F_1)\mathscr{D}(i_2,j_2,F_2)$ тоді і лише тоді, коли  $F_1=F_2$;
  \item[$\boldsymbol{(v)}$] $(i_1,j_1,F_1)\mathscr{J}(i_2,j_2,F_2)$ тоді і лише тоді, коли  існують $k_1,k_2\in\omega$ такі, що $F_1\subseteq -k_1+F_2$ і $F_2\subseteq -k_2+F_1$.
\end{itemize}
\end{theorem}

\begin{proof}
$\boldsymbol{(i)}$  Нехай $(i_1,j_1,F_1)$ i $(i_2,j_2,F_2)$~--- $\mathscr{R}$-еквівалентні елементи напівгрупи $\boldsymbol{B}_{\omega}^{\mathscr{F}}$. Оскільки згідно з теоремою \ref{theorem-3.2}, $\boldsymbol{B}_{\mathbb{Z}}^{\mathscr{F}}$~--- інверсна напівгрупа та $(i_1,j_1,F_1)\boldsymbol{B}_{\mathbb{Z}}^{\mathscr{F}}=$ \linebreak $(i_2,j_2,F_2)\boldsymbol{B}_{\mathbb{Z}}^{\mathscr{F}}$, то з теореми~1.17 в \cite{Clifford-Preston-1961} ви\-пли\-ває, що
\begin{equation*}
  (i_1,j_1,F_1)\boldsymbol{B}_{\mathbb{Z}}^{\mathscr{F}}=(i_1,j_1,F_1)(i_1,j_1,F_1)^{-1}\boldsymbol{B}_{\mathbb{Z}}^{\mathscr{F}}= (i_1,i_1,F_1)\boldsymbol{B}_{\mathbb{Z}}^{\mathscr{F}},
\end{equation*}
\begin{equation*}
  (i_2,j_2,F_2)\boldsymbol{B}_{\mathbb{Z}}^{\mathscr{F}}=(i_2,j_2,F_2)(i_2,j_2,F_2)^{-1}\boldsymbol{B}_{\mathbb{Z}}^{\mathscr{F}}= (i_2,i_2,F_2)\boldsymbol{B}_{\mathbb{Z}}^{\mathscr{F}},
\end{equation*}
звідки $(i_1,i_1,F_1)=(i_2,i_2,F_2)$. Таким чином, виконуються рівності $i_1=i_2$ i $F_1=F_2$.

Навпаки, нехай $(i_1,j_1,F_1)$ i $(i_2,j_2,F_2)$~--- елементи напівгрупи $\boldsymbol{B}_{\mathbb{Z}}^{\mathscr{F}}$ такі, що $i_1=i_2$ i $F_1=F_2$. Тоді $(i_1,i_1,F_1)=(i_2,i_2,F_2)$. Оскільки за теоремою \ref{theorem-3.2}, $\boldsymbol{B}_{\mathbb{Z}}^{\mathscr{F}}$~--- інверсна напівгрупа, то з теореми~1.17 в \cite{Clifford-Preston-1961} випливає, що
\begin{equation*}
  (i_1,j_1,F_1)\boldsymbol{B}_{\mathbb{Z}}^{\mathscr{F}}{=}(i_1,j_1,F_1)(i_1,j_1,F_1)^{-1}\boldsymbol{B}_{\mathbb{Z}}^{\mathscr{F}}{=} (i_1,i_1,F_1)\boldsymbol{B}_{\mathbb{Z}}^{\mathscr{F}}{=}(i_2,j_2,F_2)(i_2,j_2,F_2)^{-1}\boldsymbol{B}_{\mathbb{Z}}^{\mathscr{F}}{=} (i_2,i_2,F_2)\boldsymbol{B}_{\mathbb{Z}}^{\mathscr{F}},
\end{equation*}
а отже, $(i_1,j_1,F_1)\mathscr{R}(i_2,j_2,F_2)$ в напівгрупі $\boldsymbol{B}_{\mathbb{Z}}^{\mathscr{F}}$.

\smallskip
Доведення твердження $\boldsymbol{(ii)}$ аналогічне доведенню твердження $\boldsymbol{(i)}$.

\smallskip
Твердження  $\boldsymbol{(iii)}$ випливає з $\boldsymbol{(i)}$ та $\boldsymbol{(ii)}$.

\smallskip
$\boldsymbol{(iv)}$ Нехай $(i_1,j_1,F_1)$ i $(i_2,j_2,F_2)$~--- елементи напівгрупи $\boldsymbol{B}_{\mathbb{Z}}^{\mathscr{F}}$ такі, що $(i_1,j_1,F_1)\mathscr{D}(i_2,j_2,F_2)$. Оскільки $\mathscr{R},\mathscr{L}\subseteq\mathscr{D}$  та $(i_1,i_1,F_1)\mathscr{R}(i_1,j_1,F_1)$ i $(i_2,j_2,F_2)\mathscr{L}(j_2,j_2,F_2)$ за твердженнями $\boldsymbol{(i)}$ і $\boldsymbol{(ii)}$, то $(i_1,i_1,F_1)\mathscr{D}(j_2,j_2,F_2)$. За лемою Манна (див. \cite[лема 1.1]{Munn-1966}) або за твердженням 3.2.5 в \cite{Lawson-1998} існує елемент $(i,j,F)$ напівгрупи $\boldsymbol{B}_{\mathbb{Z}}^{\mathscr{F}}$ такий, що
$(i,j,F)\cdot(i,j,F)^{-1}=(i_1,i_1,F_1)$ i $(i,j,F)^{-1}\cdot(i,j,F)=(j_2,j_2,F_2)$.
З твердження \eqref{lemma-3.1(4)} леми \ref{lemma-3.1} випливає, що
$(i,j,F)\cdot(i,j,F)^{-1}=(i,j,F)\cdot(j,i,F)=(i,i,F)$ та $(i,j,F)^{-1}\cdot(i,j,F)=(j,i,F)\cdot(i,j,F)=(j,j,F)$,
а отже, $F=F_1=F_2$.

Нехай $i_1,i_2,j_1,j_2$~--- довільні цілі числа та $F\in\mathscr{F}$. За твердженнями $\boldsymbol{(i)}$ і $\boldsymbol{(ii)}$ маємо, що $(i_1,i_1,F)\mathscr{R}(i_1,j_1,F)$ i $(i_2,j_2,F)\mathscr{L}(j_2,j_2,F)$, а з твердження \eqref{lemma-3.1(2)} леми \ref{lemma-3.1} випливає, що $(i_1,i_1,F),(j_2,j_2,F)\in E(\boldsymbol{B}_{\mathbb{Z}}^{\mathscr{F}})$. Оскільки
\begin{equation*}
(i_1,j_2,F)\cdot(i_1,j_2,F)^{-1}=(i_1,j_2,F)\cdot(j_2,i_1,F)=(i_1,i_1,F)
\end{equation*}
\begin{equation*}
(i_1,j_2,F)^{-1}\cdot(i_1,j_2,F)=(j_2,i_1,F)\cdot(i_1,j_2,F)=(j_2,j_2,F),
\end{equation*}
то з леми Манна випливає, що $(i_1,i_1,F)\mathscr{D}(j_2,j_2,F)$ в $\boldsymbol{B}_{\mathbb{Z}}^{\mathscr{F}}$, і внаслідок $\mathscr{R}\circ\mathscr{D}\circ\mathscr{L}\subseteq \mathscr{D}$, виконується  $(i_1,j_1,F)\mathscr{D}(i_2,j_2,F)$.

\smallskip
$\boldsymbol{(v)}$ Зауважимо, що, оскільки $\mathscr{D}\subseteq \mathscr{J}$ і за твердженням $\boldsymbol{(iv)}$ маємо, що  $(0,0,F)\mathscr{D}(i,j,F)$ у напівгрупі $\boldsymbol{B}_{\mathbb{Z}}^{\mathscr{F}}$ для довільних $i,j\in\mathbb{Z}$ i $F\in\mathscr{F}$, то достатньо знайти необхідну та достатню умову, коли елементи $(0,0,F_1)$ i $(0,0,F_2)$ є $\mathscr{J}$-еквівалентними у напівгрупі $\boldsymbol{B}_{\mathbb{Z}}^{\mathscr{F}}$.

За твердженням~3.2.8 \cite{Lawson-1998} елементи $a$ i $b$ інверсної напівгрупи $S$ є $\mathscr{J}$-еквівалентні тоді і лише тоді, коли $a\mathscr{D}b'\preccurlyeq b$ i $b\mathscr{D}a'\preccurlyeq a$ для деяких $a',b'\in S$. За твердженням \ref{proposition-3.3} ідемпотенти $(0,0,F_1)$ i $(0,0,F_2)$ є $\mathscr{J}$-еквівалентними в напівгрупі $\boldsymbol{B}_{\mathbb{Z}}^{\mathscr{F}}$ тоді і лише тоді, коли існують  невід'ємні цілі числа $k_1$ i $k_2$ такі, що $F_1\subseteq -k_1+F_2$ і $F_2\subseteq -k_2+F_1$. З сказаного вище випливає, що $(i_1,j_1,F_1)\mathscr{J}(i_2,j_2,F_2)$ у напівгрупі $\boldsymbol{B}_{\mathbb{Z}}^{\mathscr{F}}$ тоді і лише тоді, коли існують $k_1,k_2\in\omega$ такі, що $F_1\subseteq -k_1+F_2$ і $F_2\subseteq -k_2+F_1$.
\end{proof}

Нагадаємо \cite{Lawson-1998, Ash-1979}, шо інверсна напівгрупа $S$ називається \emph{комбінаторною}, якщо відношення Ґріна $\mathscr{H}$ на $S$ є відношенням рівності.
З твердження $\boldsymbol{(iii)}$ теореми~\ref{theorem-3.5} випливає

\begin{corollary}\label{corollary-3.6}
Якщо $\mathscr{F}$ --- ${\omega}$-замкнена підсім'я в  $\mathscr{P}(\omega)$ то $\boldsymbol{B}_{\mathbb{Z}}^{\mathscr{F}}$ --- комбінаторна інверсна напівгрупа.
\end{corollary}

Також з твердження $\boldsymbol{(v)}$ теореми~\ref{theorem-3.5} випливає

\begin{corollary}\label{corollary-3.7}
Нехай $\mathscr{F}$ --- ${\omega}$-замкнена підсім'я в  $\mathscr{P}(\omega)$ і $\varnothing\notin \mathscr{F}$. Тоді напівгрупа $\boldsymbol{B}_{\mathbb{Z}}^{\mathscr{F}}$ є простою тоді і лише тоді, коли для довільних $F_1,F_2\in \mathscr{F}$ існують $k_1,k_2\in\omega$ такі, що $F_1\subseteq -k_1+F_2$ і $F_2\subseteq -k_2+F_1$.
\end{corollary}

Нагадаємо \cite{Clifford-Preston-1961}, шо напівгрупа $S$ з нулем $0$ називається \emph{$0$-простою}, якщо $S\cdot S\neq \{0\}$ i $\{0\}$ --- єдиний власний двобічний ідеал в $S$. Добре відомо (див. \cite[лема~2.28]{Clifford-Preston-1961}), що напівгрупа $S$ з нулем $0$ є $0$-простою тоді і лише тоді, коли $S$ має лише два $\mathscr{J}$-класи: $S\setminus\{0\}$ i $\{0\}$. З твердження $\boldsymbol{(v)}$ теореми~\ref{theorem-3.5} випливає

\begin{corollary}\label{corollary-3.8}
Нехай $\mathscr{F}$ --- ${\omega}$-замкнена підсім'я в  $\mathscr{P}(\omega)$ і $\varnothing\in \mathscr{F}$. Тоді напівгрупа $\boldsymbol{B}_{\omega}^{\mathscr{F}}$ є $0$-простою тоді і лише тоді, коли для довільних непорожніх множин $F_1,F_2\in \mathscr{F}$ існують $k_1,k_2\in\omega$ такі, що
$F_1\subseteq -k_1+F_2$ і $F_2\subseteq -k_2+F_1$.
\end{corollary}

Нагадаємо \cite{Saito-1965} (див. також \cite{Lawson-1998}), що інверсна напівгрупа $S$ називається \emph{$E$-унітарною}, якщо для $e\in E(S)$ i $s\in S$ з $e\preccurlyeq s$ випливає, що $s\in E(S)$. Тоді з твердження \ref{proposition-3.3} випливає

\begin{corollary}\label{corollary-3.9}
Якщо $\mathscr{F}$ --- ${\omega}$-замкнена підсім'я в  $\mathscr{P}(\omega)$ і $\varnothing\notin \mathscr{F}$, то  $\boldsymbol{B}_{\mathbb{Z}}^{\mathscr{F}}$~--- $E$-унітарна інверсна напівгрупа.
\end{corollary}

\begin{proposition}\label{proposition-3.10}
Якщо $\mathscr{F}$ --- ${\omega}$-замкнена підсім'я в  $\mathscr{P}(\omega)$, то  $\boldsymbol{B}_{\mathbb{Z}}^{\mathscr{F}}$ містить одиницю тоді і лише тоді, коли $\mathscr{F}=\{\varnothing\}$.
\end{proposition}

\begin{proof}
$(\Longrightarrow)$ Припустимо, що напівгрупа $\boldsymbol{B}_{\mathbb{Z}}^{\mathscr{F}}$ містить одиницю. Оскільки одиниця кожної напівгрупи є ідемпотентом, то, згідно з твердженням \eqref{lemma-3.1(2)} леми \ref{lemma-3.1} одиниця в $\boldsymbol{B}_{\mathbb{Z}}^{\mathscr{F}}$ має вигляд $(i,i,F)$ для деяких $i\in\mathbb{Z}$ i $F\in \mathscr{F}$. Якщо $\mathscr{F}=\{\varnothing\}$, то твердження очевидне, а тому надалі вважатимемо, що $F\neq\{\varnothing\}$. Тоді
\begin{equation*}
  (i,i,F)\cdot(i-1,i-1,F)=(i,i,F\cap(-1+F))\neq (i-1,i-1,F),
\end{equation*}
а отже елемент $(i,i,F)$ не є одиницею в $\boldsymbol{B}_{\mathbb{Z}}^{\mathscr{F}}$, протиріччя. З отриманої суперечності випливає, що $\mathscr{F}=\{\varnothing\}$.

Імплікація $(\Longleftarrow)$ очевидна.
\end{proof}

\begin{proposition}\label{proposition-3.11}
Напівгрупа $\boldsymbol{B}_{\mathbb{Z}}^{\mathscr{F}}$ ізоморфна розширеній біциклічній напівгрупі $\mathscr{C}_\omega$ тоді і тільки тоді, коли $\mathscr{F}=\{F\}$ i $F$~--- непорожня індуктивна підмножина в $\omega$.
\end{proposition}

\begin{proof}
$(\Longleftarrow)$ Нехай $F$~--- індуктивна підмножина в $\omega$ і $\mathscr{F}=\{F\}$. Для довільних $i,j,k,l\in\mathbb{Z}$ маємо, що
\begin{align*}
  (i,j,F){\cdot} (k,l,F)&{=}
\left\{
  \begin{array}{ll}
    (i-j+k,l,(j-k+F)\cap F), & \hbox{якщо~} j<k; \\
    (i,l,F\cap F),           & \hbox{якщо~} j=k; \\
    (i,j-k+l,F\cap (k-j+F)), & \hbox{якщо~} j>k
  \end{array}
\right.{=}
\left\{
  \begin{array}{ll}
    (i-j+k,l,F), & \hbox{якщо~} j<k; \\
    (i,l,F), & \hbox{якщо~} j=k; \\
    (i,j-k+l,F), & \hbox{якщо~} j>k,
  \end{array}
\right.
\end{align*}
і з означення напівгрупової операції на розширеній біциклічній напівгрупі $\mathscr{C}_\omega$ випливає, що ві\-доб\-раження $\mathfrak{f}\colon \boldsymbol{B}_{\mathbb{Z}}^{\mathscr{F}}\to \mathscr{C}_\omega$, означене за формулою $\mathfrak{f}(i,j,F)=(i,j)$, є ізоморфізмом.

\smallskip
$(\Longrightarrow)$ Припустимо, що напівгрупа $\boldsymbol{B}_{\mathbb{Z}}^{\mathscr{F}}$ ізоморфна розширеній біциклічній напівгрупі $\mathscr{C}_\omega$. Оскільки за теоремою 2.3 в~\cite{Warne-1968} розширена біциклічна напівгрупа $\mathscr{C}_\omega$ є біпростою, то з твердження $\boldsymbol{(iv)}$ теореми \ref{theorem-3.5} випливає, що  сім'я $\mathscr{F}$ складається з однієї множини $F$. З означення напівгрупової операції в $\boldsymbol{B}_{\mathbb{Z}}^{\mathscr{F}}$ випливає, що $F$ --- нескінченна підмножина в $\omega$. Справді, якщо $F$ --- скінченна підмножина в $\omega$, то
\begin{equation*}
  (0,0,F)\cdot(1,1,F)=(1,1,(-1+F)\cap F)\notin \boldsymbol{B}_{\mathbb{Z}}^{\mathscr{F}},
\end{equation*}
оскільки $(-1+F)\cap F\neq F$. У випадку, коли $F$ --- нескінченна неіндуктивна підмножина в $\omega$, то з леми 6  в \cite{Gutik-Mykhalenych-2020} випливає, що $(-1+F)\cap F\neq F$, а тому
\begin{equation*}
  (0,0,F)\cdot(1,1,F)=(1,1,(-1+F)\cap F)\notin \boldsymbol{B}_{\mathbb{Z}}^{\mathscr{F}}.
\end{equation*}
Отже, $F$~--- непорожня індуктивна підмножина в $\omega$.
\end{proof}

З твердження~\ref{proposition-3.11} випливає

\begin{corollary}\label{proposition-3.12}
Нехай $\mathscr{F}$ --- ${\omega}$-замкнена підсім'я в  $\mathscr{P}(\omega)$. Якщо сім'я $\mathscr{F}$ містить непорожню індуктивну підмножину в $\omega$, то напівгрупа $\boldsymbol{B}_{\mathbb{Z}}^{\mathscr{F}}$ містить ізоморфну копію розширеної біциклічної напівгрупи.
\end{corollary}

\begin{theorem}\label{theorem-3.13}
Нехай $\mathscr{F}$ --- ${\omega}$-замкнена підсім'я в  $\mathscr{P}(\omega)$. Тоді такі умови екві\-ва\-лентні:
\begin{itemize}
  \item[$\boldsymbol{(i)}$] $\boldsymbol{B}_{\mathbb{Z}}^{\mathscr{F}}$~--- біпроста напівгрупа;
  \item[$\boldsymbol{(ii)}$] $\mathscr{F}$ --- одноелементна сім'я;
  \item[$\boldsymbol{(iii)}$] напівгрупа $\boldsymbol{B}_{\mathbb{Z}}^{\mathscr{F}}$ або тривіальна, або ізоморфна розширеній біциклічній напівгрупі.
\end{itemize}
\end{theorem}

\begin{proof}
Еквівалентність умов $\boldsymbol{(i)}$ і $\boldsymbol{(ii)}$ випливає з твердження $\boldsymbol{(iv)}$ теореми \ref{theorem-3.5}.
Імплікація $\boldsymbol{(iii)}\Longrightarrow\boldsymbol{(ii)}$ випливає з твердження \ref{proposition-3.11}.

\smallskip
Доведемо імплікацію $\boldsymbol{(ii)}\Longrightarrow\boldsymbol{(iii)}$. Якщо $\mathscr{F}=\{\varnothing\}$, то з означення напівгрупи $\boldsymbol{B}_{\mathbb{Z}}^{\mathscr{F}}$ випливає, що $\boldsymbol{B}_{\mathbb{Z}}^{\mathscr{F}}$ --- тривіальна (одноелементна) напівгрупа. Тому припустимо, що $\mathscr{F}=\{F\}$ для деякої непорожньої множини $F$. Оскільки сім'я $\mathscr{F}$ --- ${\omega}$-замкнена, то
$(-1+F)\cap F=F$, а отже, за лемою 6 в \cite{Gutik-Mykhalenych-2020}, $F$~--- індуктивна підмножина в $\omega$. Далі скористаємося твердженням \ref{proposition-3.11}.
\end{proof}

Якщо $\lambda$~--- ненульовий кардинал, то множина $\boldsymbol{\mathcal{B}}_\lambda=(\lambda\times\lambda)\sqcup\{0\}$ з напівгруповою операцією
\begin{equation*}
  (a,b)\cdot(c,d)=
\left\{
  \begin{array}{cl}
    (a,d), & \hbox{якщо~} b=c; \\
    0,     & \hbox{якщо~} b\neq c,
  \end{array}
\right.
\qquad \hbox{i} \qquad (a,b)\cdot 0=0\cdot(a,b)=0\cdot 0=0,
\end{equation*}
де $a,b,c,d\in\lambda$, називається \emph{напівгрупою $\lambda{\times}\lambda$-матричних одиниць} \cite{Lawson-1998, Petrich-1984}.

\begin{proposition}\label{proposition-3.14}
Нехай $\mathscr{F}$ --- ${\omega}$-замкнена підсім'я в  $\mathscr{P}(\omega)$.
Напівгрупа $\boldsymbol{B}_{\mathbb{Z}}^{\mathscr{F}}$ ізо\-морфна напівгрупі $\omega{\times}\omega$-матричних одиниць $\boldsymbol{\mathcal{B}}_\omega$ тоді і тільки тоді, коли $\mathscr{F}=\{F,\varnothing\}$, де $F$~--- одноточкова підмножина в $\omega$.
\end{proposition}

\begin{proof}
$(\Longleftarrow)$
Припустимо, що $\mathscr{F}=\{F,\varnothing\}$ i $F$~--- одноточкова підмножина в $\omega$. Для довільних $i,j,k,l\in\mathbb{Z}$ маємо
\begin{equation*}
  (i,j,F)\cdot (k,l,F)=
\left\{
  \begin{array}{ll}
    (i-j+k,l,(j-k+F)\cap F), & \hbox{якщо~} j<k; \\
    (i,l,F\cap F),            & \hbox{якщо~} j=k\\
    (i,j-k+l,F\cap (k-j+F)), & \hbox{якщо~} j>k
  \end{array}
\right.=
\left\{
  \begin{array}{cl}
    (i,l,F),  & \hbox{якщо~} j=k \\
    0,        & \hbox{якщо~} j\neq k,
  \end{array}
\right.
\end{equation*}
і очевидно, що
$ 
(i,j,F)\cdot 0=0\cdot(i,j,F)=0\cdot 0=0,
$ 
звідки випливає, що відображення $\mathfrak{f}\colon \boldsymbol{B}_{\mathbb{Z}}^{\mathscr{F}}\to \boldsymbol{\mathcal{B}}_\omega$, означене за формулою  $\mathfrak{f}(i,j,F)=(i,j)$ є ізоморфізмом.

\smallskip
$(\Longrightarrow)$
Припустимо, що напівгрупа $\boldsymbol{B}_{\mathbb{Z}}^{\mathscr{F}}$ ізоморфна напівгрупі $\omega{\times}\omega$-матричних одиниць $\boldsymbol{\mathcal{B}}_\omega$. Зафіксуємо довільний ненульовий елемент $(i,j,F)$ напівгрупи $\boldsymbol{B}_{\mathbb{Z}}^{\mathscr{F}}$. Якщо $(i,j,F)$ не є ідемпотентом, то за твердженням \eqref{lemma-3.1(2)} леми~\ref{lemma-3.1} маємо, що $i\neq j$, а тоді
\begin{equation*}
  0=(i,j,F)\cdot(i,j,F)=
\left\{
  \begin{array}{ll}
    (i-j+i,j,(j-i+F)\cap F), & \hbox{якщо~} j<i; \\
    (i,j-i+j,F\cap (i-j+F)), & \hbox{якщо~} j>i.
  \end{array}
\right.
\end{equation*}
Отже, для довільних різних $i,j\in\mathbb{Z}$ маємо, що
$ 
(j-i+F)\cap F=\varnothing=F\cap (i-j+F),
$ 
а це означає, що $(-k+F)\cap F=\varnothing$ для довільного натурального числа $k$. Звідси випливає, що множина $F$ одноелементна.

Припустимо, що сім'я $\mathscr{F}$ містить дві одноелементні множини $F_k=\{k\}$ i $F_l=\{l\}$, для деяких різних $k,l\in \omega$. Не зменшуючи загальності можемо вважати, що $k<l$. Згідно з твердженням\eqref{lemma-3.1(2)} леми~\ref{lemma-3.1}, для довільного $i\in\mathbb{Z}$ елементи $(i+k,i+k,F_k)$ i $(i+l,i+l,F_l)$ є ідемпотентами в $\boldsymbol{B}_{\mathbb{Z}}^{\mathscr{F}}$, причому $(i+k,i+k,F_k)\neq(i+l,i+l,F_l)$, оскільки $k\neq l$. Оскільки напівгрупи $\boldsymbol{B}_{\mathbb{Z}}^{\mathscr{F}}$ і  $\boldsymbol{\mathcal{B}}_\omega$ є ізоморфними, то всі ненульові ідемпотенти в $\boldsymbol{B}_{\mathbb{Z}}^{\mathscr{F}}$ примітивні, а отже маємо, що
\begin{equation*}
  0=(i+k,i+k,F_k)\cdot(i+l,i+l,F_l)
   =(i+l,i+l,(k-l+F_l)\cap F_k)
   =(i+l,i+l,F_k)\neq 0,
\end{equation*}
суперечність. З отриманої суперечності випливає, що сім'я $\mathscr{F}$ містить лише одну одноелементну множину.
\end{proof}

Нехай $\boldsymbol{B}_{\omega}^{\mathscr{F}}$ --- напівгрупа, означена в \cite{Gutik-Mykhalenych-2020}. З теореми \ref{theorem-3.13} випливає, що для фіксованої сім'ї $\mathscr{F}$ напівгрупи $\boldsymbol{B}_{\mathbb{Z}}^{\mathscr{F}}$ і $\boldsymbol{B}_{\omega}^{\mathscr{F}}$ неізоморфні. Однак, з твердження 4 в \cite{Gutik-Mykhalenych-2020} і твердження \ref{proposition-3.14} випливає

\begin{corollary}\label{corollary-3.15}
Якщо $F$~--- одноточкова підмножина в $\omega$ і $\mathscr{F}=\{F,\varnothing\}$, то напівгрупи $\boldsymbol{B}_{\mathbb{Z}}^{\mathscr{F}}$ і $\boldsymbol{B}_{\omega}^{\mathscr{F}}$ ізоморфні.
\end{corollary}

Нагадаємо \cite{Lawson-1998}, шо інверсна напівгрупа $S$ з нулем $0$ називається \emph{$0$-біпростою}, якщо $S$ має лише два $\mathscr{D}$-класи: $S\setminus\{0\}$ i $\{0\}$.

Для довільного натурального числа $n$ позначимо $n\omega=\{n\cdot i\colon i\in\omega\}$.

\begin{example}\label{example-3.16}
Зафіксуємо довільні $i_0\in\omega$ та натуральне число $j_0$. Покладемо $\boldsymbol{B}_{\mathbb{Z}}^{(i_0,j_0)}=\boldsymbol{B}_{\mathbb{Z}}^{\mathscr{F}}$, де $\mathscr{F}=\left\{\varnothing,i_0+j_0\omega\right\}$. Тоді, очевидно, що $\mathscr{F}$~--- $\omega$~-замкнена сім'я в  $\mathscr{P}(\omega)$, а також за теоремою~\ref{theorem-3.2} інверсна напівгрупа $\boldsymbol{B}_{\mathbb{Z}}^{(i_0,j_0)}$ є $0$-біпростою. Більше того, для довільного $i_0\in\omega$ за твердженням \ref{proposition-3.11} напівгрупа $\boldsymbol{B}_{\mathbb{Z}}^{(i_0,1)}$ ізоморфна розширеній біциклічній напівгрупі з приєднаним нулем.
\end{example}

Безпосередньо звичайною перевіркою доводиться, що для довільних $i_1,i_2\in\omega$ та довільного натурального числа $j_0$ відображення $\mathfrak{h}\colon \boldsymbol{B}_{\mathbb{Z}}^{(i_1,j_0)}\to \boldsymbol{B}_{\mathbb{Z}}^{(i_2,j_0)}$, означене
\begin{equation*}
  \mathfrak{h}(n,m,i_1+j_0\omega)=(n,m,i_2+j_0\omega) \qquad \hbox{i} \qquad \mathfrak{h}(0)=0
\end{equation*}
є ізоморфізмом. Отже, виконується

\begin{proposition}\label{proposition-3.17}
Для довільних $i_1,i_2\in\omega$ та довільного натурального числа $j_0$ напівгрупи $\boldsymbol{B}_{\mathbb{Z}}^{(i_1,j_0)}$ i $\boldsymbol{B}_{\mathbb{Z}}^{(i_2,j_0)}$ ізоморфні.
\end{proposition}

Наступна теорема описує структуру $0$-біпростих інверсних напівгруп $\boldsymbol{B}_{\mathbb{Z}}^{\mathscr{F}}$ з точністю до ізоморфізму.

\begin{theorem}\label{theorem-3.18}
Нехай $\mathscr{F}$ --- ${\omega}$-замкнена підсім'я в  $\mathscr{P}(\omega)$, $\varnothing\in\mathscr{F}$ i $\boldsymbol{B}_{\mathbb{Z}}^{\mathscr{F}}$~--- $0$-біпроста напівгрупа. Тоді виконується лише одна з умов:
\begin{itemize}
  \item[$(1)$] напівгрупа $\boldsymbol{B}_{\mathbb{Z}}^{\mathscr{F}}$ ізоморфна напівгрупі $\omega{\times}\omega$-матричних одиниць $\boldsymbol{\mathcal{B}}_\omega$;
  \item[$(2)$] напівгрупа $\boldsymbol{B}_{\mathbb{Z}}^{\mathscr{F}}$ ізоморфна напівгрупі $\boldsymbol{B}_{\omega}^{(0,j_0)}$ для деякого натурального числа $j_0$.
\end{itemize}
\end{theorem}

\begin{proof}
За твердженням $\boldsymbol{(iv)}$ теореми \ref{theorem-3.5} сім'я $\mathscr{F}$ містить непорожню множину $F$ та порож\-ню множину $\varnothing$. Припустимо, що множина $F$ скінченна. Тоді, аналогічно, як і в твердженні \ref{proposition-3.14} доводиться, що $F$ --- одноелементна множина, а отже за твердженням \ref{proposition-3.14}, напівгрупа $\boldsymbol{B}_{\mathbb{Z}}^{\mathscr{F}}$ ізоморфна напівгрупі $\omega{\times}\omega$-матричних одиниць $\boldsymbol{\mathcal{B}}_\omega$.

Якщо ж $F$ --- нескінченна множина, то з теореми \ref{theorem-3.5}$\boldsymbol{(iv)}$ випливає, що викону\-ється одна з умов:
\begin{equation*}
(a)~(-1+F)\cap F=F \qquad \hbox{або} \qquad (b)~(-1+F)\cap F=\varnothing.
\end{equation*}

У випадку $(a)$ за твердженням \ref{proposition-3.11} множина $\boldsymbol{B}_{\mathbb{Z}}^{\mathscr{F}}\setminus\{0\}$ є піднапівгрупою в $\boldsymbol{B}_{\mathbb{Z}}^{\mathscr{F}}$, яка ізоморфна розширеній біциклічній напівгрупі, а отже виконується твердження $(2)$.

Якщо ж виконується умова $(b)$, то з леми 7 \cite{Gutik-Mykhalenych-2020} випливає, що $F=i_0+j_0\omega$ для деяких натурального числа $j_0$ та $i_0\in\omega$. Застосувавши твердження \ref{proposition-3.17}, отримуємо, що виконується твердження $(2)$.
\end{proof}

Нагадаємо \cite{Lawson-1998, Petrich-1984}, що \emph{найменша} (\emph{мінімальна}) \emph{групова конгруенція} $\boldsymbol{\sigma}$ на ін\-версній напівгрупі $S$ визначається так:
\begin{equation*}
  s\boldsymbol{\sigma}t \qquad \Longleftrightarrow \qquad es=et \quad \hbox{для деякого} \quad e\in E(S).
\end{equation*}

Очевидно, що, якщо $\mathscr{F}$ --- ${\omega}$-замкнена підсім'я в  $\mathscr{P}(\omega)$ i $\varnothing\in\mathscr{F}$, то напівгрупа $\boldsymbol{B}_{\mathbb{Z}}^{\mathscr{F}}$ містить нуль, а отже фактор-напівгрупа $\boldsymbol{B}_{\mathbb{Z}}^{\mathscr{F}}/\boldsymbol{\sigma}$ ізоморфна тривіальній групі.

\begin{proposition}\label{proposition-3.19}
Нехай $\mathscr{F}$ --- ${\omega}$-замкнена підсім'я в  $\mathscr{P}(\omega)$ i $\varnothing\notin\mathscr{F}$. Тоді  $(i_1,j_1,F_1)\boldsymbol{\sigma}(i_2,j_2,F_2)$ в $\boldsymbol{B}_{\mathbb{Z}}^{\mathscr{F}}$ тоді і тільки тоді, коли $i_1-j_1=i_2-j_2$, а отже фактор-напівгрупа $\boldsymbol{B}_{\mathbb{Z}}^{\mathscr{F}}/\boldsymbol{\sigma}$ ізоморфна адитивній групі цілих чисел $\mathbb{Z}(+)$.
\end{proposition}

\begin{proof}
Нехай $(i_1,j_1,F_1)$ i $(i_2,j_2,F_2)$ --- довільні  елементи напівгрупи $\boldsymbol{B}_{\mathbb{Z}}^{\mathscr{F}}$. З означення найменшої групової конгруенції $\boldsymbol{\sigma}$ випливає, що $(i_1,j_1,F_1)\boldsymbol{\sigma}(i_2,j_2,F_2)$ тоді і тільки тоді, коли існує елемент $(i,j,F)\in\boldsymbol{B}_{\mathbb{Z}}^{\mathscr{F}}$ такий, що $(i,j,F)\preccurlyeq(i_1,j_1,F_1)$ i $(i,j,F)\preccurlyeq(i_2,j_2,F_2)$. Тоді за твердженням~\ref{proposition-3.3} маємо, що
$F\subseteq -k_1+F_1$ i $i-i_1=j-j_1=k_1$, а також
$F\subseteq -k_2+F_2$ i $i-i_2=j-j_2=k_2$
для деяких $k_1,k_2\in\omega$. Отже, з $(i_1,j_1,F_1)\boldsymbol{\sigma}(i_2,j_2,F_2)$ в $\boldsymbol{B}_{\mathbb{Z}}^{\mathscr{F}}$ випливає, що
$
i_1-j_1=i_2-j_2=i-j.
$

Припустимо, що для елементів $(i_1,j_1,F_1)$ i $(i_2,j_2,F_2)$  напівгрупи $\boldsymbol{B}_{\mathbb{Z}}^{\mathscr{F}}$ справджується рівність $i_1-j_1=i_2-j_2$. Не зменшуючи загальності можемо вважати, що $i_1>i_2$. Позаяк $\mathscr{F}$ --- ${\omega}$-замкнена підсім'я в  $\mathscr{P}(\omega)$, то
$
F=F_1\cap (i_2-i_1+F_2)\in \mathscr{F}.
$
Тоді $j_1>j_2$ і за твердженням~\ref{proposition-3.3} маємо, що
\begin{align*}
  (i_1,j_1,F)&=(i_1,j_1,F\cap F_1)=(i_1,i_1,F)\cdot(i_1,j_1,F_1) \preccurlyeq(i_1,j_1,F_1)
\end{align*}
i
\begin{align*}
    (i_1,i_1,F)\cdot (i_2,j_2,F_2)&=(i_1,i_1-i_2+j_2,F\cap(i_2-i_1+F_2))
                                  =(i_1,j_1-j_2+j_2,F\cap(i_2-i_1+F_2))=\\
                                  &=(i_1,j_1,F\cap(i_2-i_1+F_2))
                                  =(i_1,j_1,F)
                                  \preccurlyeq(i_2,j_2,F_2),
\end{align*}
а отже, $(i_1,j_1,F_1)\boldsymbol{\sigma}(i_2,j_2,F_2)$.

Означимо відображення $\mathfrak{h}_{\boldsymbol{\sigma}}\colon \boldsymbol{B}_{\mathbb{Z}}^{\mathscr{F}} \to Z(+)$ за формулою $\mathfrak{h}_{\boldsymbol{\sigma}}(i,j,F)=i-j$. Із доведеного вище  випливає, що $\mathfrak{h}_{\boldsymbol{\sigma}}(i_1,j_1,F_1)=\mathfrak{h}_{\boldsymbol{\sigma}}(i_2,j_2,F_2)$ тоді і лише тоді, коли $(i_1,j_1,F_1)\boldsymbol{\sigma}(i_2,j_2,F_2)$ в $\boldsymbol{B}_{\mathbb{Z}}^{\mathscr{F}}$, а отже відображення $\mathfrak{h}_{\boldsymbol{\sigma}}\colon \boldsymbol{B}_{\mathbb{Z}}^{\mathscr{F}} \to Z(+)$ є гомоморфізмом і фактор-напівгрупа $\boldsymbol{B}_{\mathbb{Z}}^{\mathscr{F}}/\boldsymbol{\sigma}$ ізоморфна адитивній групі цілих чисел $\mathbb{Z}(+)$.
\end{proof}

На завершенні ми опишемо структуру напівгрупи $\boldsymbol{B}_{\mathbb{Z}}^{\mathscr{F}}$, де сім'я $\mathscr{F}$ складається з усіх одноелементних підмножин в $\omega$ та порожньої множини. Зауважимо, що в \cite{Lysetska-2020} і \cite{Gutik-Lysetska-2021?} отримано подібні результати для напівгруп  $\boldsymbol{B}_{\omega}^{\mathscr{F}}$ і $\boldsymbol{B}_{\omega}^{\mathscr{F}_1}$, де сім'ї $\mathscr{F}$ і $\mathscr{F}_1$ складаються з усіх одноелементних підмножин в $\omega$ і порожньої множини та атомарних підмножин у $\omega$, відповідно.

Приймемо
$
  \mathscr{F}_1=\left\{A\subseteq \omega\colon |A|\leqslant 1\right\}.
$
Очевидно, що $\mathscr{F}_1$ --- ${\omega}$-замкнена підсім'я в $\mathscr{P}(\omega)$, а отже $\boldsymbol{B}_{\mathbb{Z}}^{\mathscr{F}_1}$ --- інверсна напівгрупа з нулем. Далі через $(i,j,\{k\})$ будемо позначати ненульовий елемент напівгрупи $\boldsymbol{B}_{\mathbb{Z}}^{\mathscr{F}_1}$ для деяких $i,j\in\mathbb{Z}$ i $k\in\omega$, а через $\boldsymbol{0}$ нуль напівгрупи $\boldsymbol{B}_{\mathbb{Z}}^{\mathscr{F}_1}$.

Скористаємося конструкцією з  \cite{Gutik-1999}.
Нехай $S$ --- напівгрупа та $X$~--- непорожня множина потуж\-ності $\lambda$. На множині
$B_{X}(S)=\left(X\times S\times X\right)\sqcup\{ \mathscr{O}\}$ означимо напівгрупову операцію так:
\begin{equation*}
 (\alpha,s,\beta)\cdot(\gamma, t, \delta)=
\left\{
  \begin{array}{cl}
    (\alpha, st, \delta), & \hbox{якщо~} \beta=\gamma; \\
    \mathscr{O},          & \hbox{якщо~} \beta\ne \gamma
  \end{array}
\right.
\end{equation*}
і $(\alpha, s, \beta)\cdot \mathscr{O}=\mathscr{O}\cdot(\alpha, s, \beta)=\mathscr{O}\cdot \mathscr{O}=\mathscr{O}$, для всіх $\alpha, \beta, \gamma, \delta\in X$ i $s, t\in S$. Якщо $S$ --- моноїд, то напівгрупа $\mathscr{B}_X(S)$ називається {\it $\lambda$-розширенням Брандта напівгрупи} $S$~\cite{Gutik-1999}. Властивості напівгрупи $\mathscr{B}_\lambda(S)$ та її узагальнення  $\lambda^0$-розширення Брандта $\mathscr{B}^0_X(S)$ напівгруп вивчалися в \cite{Gutik-1999, Gutik-Pavlyk-2006, Gutik-Repovs-2010}.

Через $\omega_{\min}$ позначимо множину $\omega$ з бінарною операцією
\begin{equation*}
  xy=\min\{x,y\}, \qquad \hbox{для} \quad x,y\in \mathbb{Z}.
\end{equation*}
Очевидно, що $\omega_{\min}$ -- напівґратка.

Означимо відображення $\mathfrak{f}\colon\boldsymbol{B}_{\mathbb{Z}}^{\mathscr{F}_1}\to\mathscr{B}_{\mathbb{Z}}(\omega_{\min})$ так:
\begin{equation}\label{eq-1}
  \mathfrak{f}(i,j,\{k\})=(i+k,k,j+k) \qquad \hbox{і} \qquad (\boldsymbol{0})\mathfrak{f}=\mathscr{O},
\end{equation}
для $i,j\in\mathbb{Z}$ i $k\in\omega$.

\begin{theorem}\label{theorem-3.20}
Напівгрупа  $\boldsymbol{B}_{\mathbb{Z}}^{\mathscr{F}_1}$ ізоморфна напівгрупі $\mathscr{B}_{\mathbb{Z}}(\omega_{\min})$ стосовно відображення $\mathfrak{f}$.
\end{theorem}

\begin{proof}
Очевидно, що відображення $\mathfrak{f}$, визначене за формулою \eqref{eq-1}, є бієктивним.

Зафіксуємо довільні $(i_1,j_1,\{k_1\}),(i_2,j_2,\{k_2\})\in \boldsymbol{B}_{\mathbb{Z}}^{\mathscr{F}_1}$. Тоді
\begin{align*}
  \mathfrak{f}((i_1,j_1,\{k_1\})&\cdot(i_2,j_2,\{k_2\}))=\\
&=\left\{
  \begin{array}{cl}
    \mathfrak{f}(i_1-j_1+i_2,j_2,(j_1-i_2+\{k_1\})\cap\{k_2\}), & \hbox{якщо~} j_1<i_2 \hbox{~і~} j_1+k_1=i_2+k_2;\\
    \mathfrak{f}(\boldsymbol{0}),                               & \hbox{якщо~} j_1<i_2 \hbox{~і~} j_1+k_1\neq i_2+k_2;\\
    \mathfrak{f}(i_1,j_2,\{k_1\}\cap\{k_2\}),                   & \hbox{якщо~} j_1=i_2 \hbox{~і~} k_1=k_2;\\
    \mathfrak{f}(\boldsymbol{0}),                               & \hbox{якщо~} j_1=i_2 \hbox{~і~} k_1\neq k_2;\\
    \mathfrak{f}(i_1,j_1-i_2+j_2,\{k_1\}\cap(i_2-j_1+\{k_2\})), & \hbox{якщо~} j_1>i_2 \hbox{~і~} j_1+k_1=i_2+k_2;\\
    \mathfrak{f}(\boldsymbol{0}),                               & \hbox{якщо~} j_1>i_2 \hbox{~і~} j_1+k_1\neq i_2+k_2
  \end{array}
\right. =\\
    & =\left\{
  \begin{array}{cl}
    \mathfrak{f}(i_1-j_1+i_2,j_2,\{k_2\}), & \hbox{якщо~} j_1<i_2 \hbox{~і~} j_1+k_1=i_2+k_2;\\
    \mathfrak{f}(i_1,j_2,\{k_1\}),         & \hbox{якщо~} j_1=i_2 \hbox{~і~} k_1=k_2;\\
    \mathfrak{f}(i_1,j_1-i_2+j_2,\{k_1\}), & \hbox{якщо~} j_1>i_2 \hbox{~і~} j_1+k_1=i_2+k_2;\\
    \mathfrak{f(\boldsymbol{0})},          & \hbox{якщо~} j_1+k_1\neq i_2+k_2
  \end{array}
\right.=\\
    & =\left\{
  \begin{array}{cl}
    (i_1-j_1+i_2+k_2,k_2,j_2+k_2), & \hbox{якщо~} j_1<i_2 \hbox{~і~} j_1+k_1=i_2+k_2;\\
    (i_1+k_1,k_1,j_2+k_1),         & \hbox{якщо~} j_1=i_2 \hbox{~і~} k_1=k_2;\\
    (i_1+k_1,k_1,j_1-i_2+j_2+k_1), & \hbox{якщо~} j_1>i_2 \hbox{~і~} j_1+k_1=i_2+k_2;\\
    \mathscr{O},                   & \hbox{якщо~} j_1+k_1\neq i_2+k_2
  \end{array}
\right.=\\
    & =\left\{
  \begin{array}{cl}
    (i_1+k_1,k_2,j_2+k_2), & \hbox{якщо~} j_1<i_2 \hbox{~і~} j_1+k_1=i_2+k_2;\\
    (i_1+k_1,k_1,j_2+k_2), & \hbox{якщо~} j_1=i_2 \hbox{~і~} k_1=k_2;\\
    (i_1+k_1,k_1,j_2+k_2), & \hbox{якщо~} j_1>i_2 \hbox{~і~} j_1+k_1=i_2+k_2;\\
    \mathscr{O},           & \hbox{якщо~} j_1+k_1\neq i_2+k_2,
  \end{array}
\right.
\end{align*}
і
\begin{align*}
  \mathfrak{f}(i_1,j_1,\{k_1\})\cdot\mathfrak{f}(i_2,j_2,\{k_2\})&=(i_1+k_1,k_1,j_1+k_1)\cdot(i_2+k_2,k_2,j_2+k_2)=\\
&=
\left\{
  \begin{array}{cl}
    (i_1+k_1,\min\{k_1,k_2\},j_2+k_2), & \hbox{якщо~} j_1+k_1=i_2+k_2; \\
    \mathscr{O},                       & \hbox{якщо~} j_1+k_1\neq i_2+k_2
  \end{array}
\right.=
\end{align*}
\begin{align*}
    & =\left\{
  \begin{array}{cl}
    (i_1+k_1,k_2,j_2+k_2), & \hbox{якщо~} k_2<k_1 \hbox{~і~} j_1+k_1=i_2+k_2;\\
    (i_1+k_1,k_1,j_2+k_2), & \hbox{якщо~} k_2=k_1 \hbox{~і~} j_1=j_2;\\
    (i_1+k_1,k_1,j_2+k_2), & \hbox{якщо~} k_2>k_1 \hbox{~і~} j_1+k_1=i_2+k_2;\\
    \mathscr{O},           & \hbox{якщо~} j_1+k_1\neq i_2+k_2,
  \end{array}
\right.=\\
    & =\left\{
  \begin{array}{cl}
    (i_1+k_1,k_2,j_2+k_2), & \hbox{якщо~} j_1<i_2 \hbox{~і~} j_1+k_1=i_2+k_2;\\
    (i_1+k_1,k_1,j_2+k_2), & \hbox{якщо~} j_1=i_2 \hbox{~і~} k_1=k_2;\\
    (i_1+k_1,k_1,j_2+k_2), & \hbox{якщо~} j_1>i_2 \hbox{~і~} j_1+k_1=i_2+k_2;\\
    \mathscr{O},           & \hbox{якщо~} j_1+k_1\neq i_2+k_2.
  \end{array}
\right.
\end{align*}
Оскільки $\boldsymbol{0}$ і $\mathscr{O}$ є нулями напівгруп $\boldsymbol{B}_{\mathbb{Z}}^{\mathscr{F}_1}$ і $\mathscr{B}_{\mathbb{Z}}(\omega_{\min})$, відповідно, то з вище наведеного випливає, що відображення
$\mathfrak{f}\colon\boldsymbol{B}_{\mathbb{Z}}^{\mathscr{F}_1}\to\mathscr{B}_{\mathbb{Z}}(\omega_{\min})$ є ізоморфізмом.
\end{proof}



\begin{thebibliography}{10}

\bibitem{Wagner-1952}
В. В.~Вагнер,
\emph{Обощенные группы},
ДАН СССР \textbf{84} (1952), 1119--1122.

\bibitem{Gutik-1999}
О. В. Гутік,
\emph{Про напівгрупу Гауі},
Мат. методи фіз.-мех. поля \textbf{42}, (1999), no. 4,  127--132

\bibitem{Gutik-Mykhalenych-2020}
О.~Гутік, М. Михаленич,
\emph{Про одне узагальнення бiциклiчного моноїда},
Вісник Львів. ун-ту. Сер. мех.-мат. \textbf{90} (2020), 5--19.

\bibitem{Andersen-1952}
O. Andersen,
\emph{Ein Bericht uber die Struk\-tur abstrakter Halbgruppen},
PhD Thesis. Ham\-burg, 1952.

\bibitem{Ash-1979}
C. J. Ash,
\emph{The $\mathscr{J}$-classes of an inverse semigroup},
J. Austral. Math. Soc. Ser. A \textbf{28} (1979), 427--432.

\bibitem{Bruck-1958}
R. H. Bruck,
\emph{A survey of binary systems},
(Erg. Math. Grenzgebiete. Neue Folge. Heft 20) Springer, Berlin-G\"{o}ttingen-Heidelberg,  1958.

\bibitem{Clifford-Preston-1961}
A.~H.~Clifford and G.~B.~Preston,
\emph{The Algebraic Theory of Semigroups}, Vol. I,
Amer. Math. Soc. Surveys {\bf 7}, Pro\-vi\-den\-ce, R.I.,  1961.


\bibitem{Clifford-Preston-1967}
A.~H.~Clifford and G.~B.~Preston,
\emph{The Algebraic Theory of Semigroups}, Vol.  II,
Amer. Math. Soc. Surveys {\bf 7}, Pro\-vi\-den\-ce, R.I.,   1967.

\bibitem{Fihel-Gutik-2011}
I. R. Fihel and O. V. Gutik,
\emph{On the closure of the extended bicyclic semigroup},
Карпатські математичні публікації \textbf{3} (2011), no. 2, 131--157.

\bibitem{Fortunatov-1976}
V. A. Fortunatov,
\emph{Congruences on simple extensions of semigroups},
Semigroup Forum \textbf{13} (1976), 283--295.

\bibitem{Fotedar-1974}
G.~L.~Fotedar, \emph{On a semigroup associated with an ordered group},
Math. Nachr. \textbf{60} (1974), 297--302.

\bibitem{Fotedar-1978}
G.~L.~Fotedar, \emph{On a class of bisimple inverse semigroups},
Riv. Mat. Univ. Parma (4) \textbf{4} (1978), 49--53.

\bibitem{Green-1951}
J. A. Green
\emph{On the structure of semigroups},
Ann. Math. Ser. 2 \textbf{54} (1951), no.~1, 163--172.

\bibitem{Gutik-2018}
O. Gutik,
\emph{On locally compact semitopological $0$-bisimple inverse $\omega$-semigroups,}
Topol. Algebra Appl. \textbf{6} (2018), 77--101.

\bibitem{Gutik-Lysetska-2021?}
O. Gutik and O. Lysetska,
\emph{On the semigroup $\boldsymbol{B}_{\omega}^{\mathscr{F}}$ which is generated by the family $\mathscr{F}$ of singletons of $\omega$}, Preprint, 2021.

\bibitem{Gutik-Pagon-Pavlyk-2011}
O. Gutik, D. Pagon, and K. Pavlyk,
\emph{Congruences on bicyclic extensions of a linearly ordered group},
Acta Comment. Univ. Tartu. Math. \textbf{15} (2011), no. 2, 61--80.

\bibitem{Gutik-Pavlyk-2006}
O. V. Gutik and K. P. Pavlyk,
\emph{On Brandt $\lambda^0$-extensions of semigroups with zero},
Мат. методи фіз.-мех. поля \textbf{49} (2006), no. 3, 26--40.

\bibitem{Gutik-Repovs-2010}
O. Gutik  and D. Repov\v{s},
\emph{On Brandt $\lambda^0$-extensions of monoids with zero},
Semigroup Forum \textbf{80} (2010), no. 1, 8--32.

\bibitem{Lawson-1998}
M.~Lawson,
\emph{Inverse Semigroups. The Theory of Partial Symmetries},
Singapore: World Scientific, 1998.

\bibitem{Lysetska-2020}
O. Lysetska,
\emph{On feebly compact topologies on the semigroup $\boldsymbol{B}_{\omega}^{\mathscr{F}_1}$},
Вісник Львів. ун-ту. Сер. мех.-мат. \textbf{90} (2020), 48--56.

\bibitem{Munn-1966}
W. D. Munn,
\emph{Uniform semilattices and bisimple inverse semigroups},
Quart. J. Math. \textbf{17} (1966),  no.~1,  151--159.

\bibitem{Petrich-1984}
M.~Petrich,
\emph{Inverse Semigroups},
John Wiley $\&$ Sons, New York, 1984.

\bibitem{Reilly-1966}
N. R. Reilly,
\emph{Bisimple $\omega$-semigroups, }
Proc. Glasgow Math. Assoc. \textbf{7} (1966), 160--167.

\bibitem{Saito-1965}
T. Sait\^{o},
\emph{Proper ordered inverse semigroups},
Pacif. J. Math. \textbf{15} (1965), no. 2, 649--666.



\bibitem{Warne-1966}
R. J. Warne,
\emph{A class of bisimple inverse semigroups},
Pacif. J. Math. \textbf{18} (1966), no. 3, 563--577.

\bibitem{Warne-1967}
R. J. Warne,
\emph{Bisimple inverse semigroups mod groups},
Duke Math. J. \textbf{34} (1967), no. 4, 787--812.

\bibitem{Warne-1968}
R. J. Warne,
\emph{$I$-bisimple semigroups},
Trans. Amer. Math. Soc.  \textbf{130} (1968), no. 3, 367--386.

\end{thebibliography}
\end{document}